\documentclass{amsart}

\usepackage[utf8]{inputenc}
\usepackage{amsmath}
\usepackage{amsthm}
\usepackage{amssymb}
\usepackage{mathtools}
\usepackage{prettyref}
\usepackage{verbatim}
\usepackage[unicode=true,pdfusetitle,
bookmarks=true,bookmarksnumbered=false,bookmarksopen=false,
breaklinks=false,pdfborder={0 0 0},pdfborderstyle={},backref=false,colorlinks=false]
{hyperref}
\usepackage[all]{xy}
\usepackage{color}

\newtheorem{thm}{Theorem}[section]
\newtheorem*{thm*}{Theorem}
\newrefformat{thm}{Theorem~\ref{#1}}
\makeatletter
\let\old@newtheorem\newtheorem
\renewcommand{\newtheorem}[2]{\old@newtheorem{#1}[thm]{#2}
\newrefformat{#1}{#2~\ref{##1}}}
\makeatother

\theoremstyle{plain}

\newtheorem{prop}{Proposition}
\newtheorem{cor}{Corollary}
\newtheorem{lem}{Lemma}

\newtheorem{fact}{Fact}

\theoremstyle{definition}
\newtheorem{defn}{Definition}

\newtheorem{notation}{Notation}

\theoremstyle{remark}
\newtheorem{rem}{Remark}

\newcommand{\R}{\mathbb R}
\newcommand{\N}{\mathbb N}

\newcommand{\K}{\mathbb K}
\newcommand{\Z}{\mathbb Z}
\newcommand{\on}{\mathbf{On}}
\newcommand{\no}{\mathbf{No}}

\DeclareMathOperator{\eps}{\varepsilon}
\DeclareMathOperator{\suchthat}{\,:\,}

\DeclareMathOperator{\rf}{\mathbf{k}}
\DeclareMathOperator{\vr}{O(1)}
\DeclareMathOperator{\vleq}{\preceq}
\DeclareMathOperator{\vless}{\prec}
\DeclareMathOperator{\veq}{\asymp}

\title{Exponential fields and Conway's omega-map}
\author[A.~Berarducci, S.~Kulhmann, V.~Mantova, M.~Matusinski]{Alessandro Berarducci, Salma Kuhlmann,\\ Vincenzo Mantova, Micka\"el Matusinski}
\thanks{A.B. was partially supported by the project ``PRIN 2012, Logica Modelli e Insiemi". V.M. was partially supported by ``Fondation Sciences Mathématiques de Paris"}
\address{Alessandro Berarducci: Dipartimento di Matematica, Universit\`a di Pisa,
        Largo B. Pontecorvo 5, 56127 Pisa, Italy}
\email{alessandro.berarducci@unipi.it}
\address{Salma Kuhlmann: Fachbereich Mathematik und Statistik,
Universität Konstanz,
Universitätsstraße 10,
78457 Konstanz, Germany}
\email{salma.kuhlmann@uni-konstanz.de}
\address{Micka\"el Matusinski: Institut de Mathématiques de Bordeaux UMR 5251,
                           Université de Bordeaux,
                           351 cours de la Libération,
                           33405 Talence cedex, France}
\email{mickael.matusinski@math.u-bordeaux.fr}
\address{Vincenzo Mantova: School of Mathematics, University of Leeds,
Leeds LS2 9JT,
United Kingdom}
\email{v.l.mantova@leeds.ac.uk}
\date{Oct. 8th, 2018. Revised Jan. 17, 2019}
\begin{document}

\begin{abstract}
 Inspired by Conway's surreal numbers, we study real closed fields whose value group is isomorphic to the additive reduct of the field. We call such fields omega-fields and we prove that any omega-field of bounded Hahn series with real coefficients admits an exponential function making it into a model of the theory of the real exponential field. We also consider relative versions with more general coefficient fields.
\end{abstract}
\maketitle

\section{Introduction}
We study real closed valued fields $\K$, with a convex valuation ring $\vr\subseteq \K$ satisfying the property that the value group $v(\K^{\times})$ is isomorphic to the additive reduct $(\K,+,<)$ of the field, where $v$ is the valuation induced by $\vr$. We call \textbf{omega-field} a field with this property. The name is motivated by the fact that any omega-field admits a map akin to Conway's omega-map $x\mapsto \omega^x$ on the field of surreal numbers $\no$ \cite{Conway76} or its fragments $\no(\lambda)$ studied in \cite{DriesE2001}, where $\lambda$ is an $\epsilon$-number. We need to recall that any real closed field $\K$ admits a section of the valuation, hence it has a multiplicative subgroup $G\subseteq \K^{>0}$, called a group of \textbf{monomials}, given by the image of the section. Since $G$ is a multiplicative copy of $v(\K^{\times})$, we have that $\K$ is an omega-field if and only if it admits an isomorphism $$\omega:(\K,+,0,<)\cong (G,\cdot,1,<),$$ and we shall call \textbf{omega-map} any such isomorphism. The prototypical example is Conway's omega-map on the surreal numbers, and in analogy with the surreal case, we use the exponential notation $\omega^x$ to denote the image of $x$ under $\omega$.

Here we explore the relation between omega-fields and exponential fields, where an \textbf{exponential field} is a real closed field $\K$ admitting an \textbf{exponential map}, that is an isomorphism $\exp:(\K,+,0,<)\cong (\K^{>0},\cdot,1,<)$. We shall freely write $e^x$ rather than $\exp(x)$ when convenient. Note that $\omega^x$ should not be read as $e^{\omega \log(x)}$ (the easiest way to see why is that the map $x \mapsto \omega^x$, if there is such an omega-map, is not continuous in the order topology of $\K$). While in general there are no containments between the class of fields admitting an omega-map and that of fields admitting an exponential map, a non-trivial inclusion of the former in the latter can be obtained by restricting the analysis to $\kappa$-bounded Hahn fields, as discussed below.

In general, any real closed valued field $\K$ with monomials $G$ is isomorphic to a truncation closed subfield (see Definition \ref{defn:analytic subfield} (1)) of the Hahn field $\rf((G))$ \cite{Mourgues1993}, where $\rf\cong \vr/o(1)$ is the residue field and we write $o(1)$ for the maximal ideal of $\vr$. For the sake of simplicity in this introduction we focus on the typical case $\rf=\R$, but our results hold more generally assuming that the residue field $\rf$ is a model of $T_{an,\exp}$,  the theory of the real exponential field $\R_{\exp}$ with all restricted analytic functions \cite{Dries1994}.
The full Hahn field $\R((G))$ is always naturally a model of the theory of restricted analytic functions $T_{an}$ \cite{Dries1994}, but it never admits an exponential function if $G\not= 1$ \cite{Kuhlmann1997}. However, for every regular uncountable cardinal $\kappa$, there is a group $G$ such that the \emph{$\kappa$-bounded} Hahn field $\R((G))_\kappa$ does admit an exponential function \cite{Kuhlmann2005}. We thus restrict our analysis to fields of the form $\K= \R((G))_\kappa$ (without assuming \emph{a priori} that they admit an exponential map). Our first result is the following. The case $G = \no(\kappa)$ with $\kappa$ regular uncountable is in \cite{DriesE2001}.

\begin{thm*}[\ref{thm:omega-exp}]
  Every omega-field of the form $\R((G))_\kappa$ admits an exponential function making it into a model of $T_{an,\exp}$.
\end{thm*}
Our work was partly motivated by the desire to understand the connections between the surreal numbers, with its various subfields studied in \cite{Berarducci2015a,Mantova2017}, and the exponential fields of the form $\R((G))_\kappa$ constructed by S.~Kuhlmann and S.~Shelah in \cite{Kuhlmann2005}.
We shall prove that the latter are not always omega-fields (\prettyref{thm:no-omega}), but they are omega-fields  if and only if  $G$ is order-isomorphic to $G^{>1}$ (\prettyref{thm:criterion-for-omega}); in this case, given a chain isomorphism $\psi:G\cong G^{>1}$, there is an omega-map satisfying $\omega^g = e^{\psi(g)}$ for all $g\in G$.

Let us now discuss Theorem \ref{thm:omega-exp} in more detail. We show that given $\K=\R((G))_\kappa$ and an omega-map $\omega:\K\cong G$, we can construct an exponential function directly starting from $\omega$ and an auxiliary chain isomorphism
$$h : (\K,<)\cong(\K^{>0},<),$$
where by \textbf{chain} we mean linearly ordered set. Any choice of $h$ yields an exponential field (\prettyref{thm:omega-log}) and at least one choice of $h$ will yield a model of $T_{an,\exp}$ (\prettyref{thm:omega-exp}). Varying $h$ we can thus produce a variety of exponential fields; some of them are models of $T_{\exp}$, while all the others are not even o-minimal (\prettyref{thm:omin}), depending on the growth properties of $h$ (\prettyref{defn:growth}).

To define the exponential function, it is more convenient to first define a \textbf{logarithm} $\log:\K^{>0}\to \K$ and let $\exp$ be the compositional inverse $\log$. To this aim we start by putting $$\log(\omega^{\omega^x})=\omega^{h(x)}$$ for $x\in \K$ and $\log(1+\eps) = \sum_{n=1}^\infty (-1)^{n+1} \frac{\eps^n}{n}$ for $\eps\in o(1)$. Note that such infinite sums make sense in the $\kappa$-bounded Hahn field $\R((G))_\kappa$.

The extension of $\log$ to the whole of $\K^{>0}$ is then carried out guided by the principle that $\log$ takes products into sums and $\omega$ takes sums into products. We simply extend this to infinite sums. More precisely, $\log$ is determined by $\log(\omega^{\sum_{i<\alpha} \omega^{\gamma_i} r_i})=\sum_{i<\alpha}\omega^{h(\gamma_i)}r_i$ and $\log(\omega^xr(1+\eps))=\log(\omega^x)+\log(r) + \sum_{n=1}^\infty (-1)^{n+1} \frac{\eps^n}{n}$, where $\eps\in o(1)$ and $r\in \R$. Another way to express the spirit of the construction is that we first define $\log$ on the representatives of the multiplicative archimedean classes $\omega^{\omega^x}$, then we extend it
to the representatives of the additive archimedean classes $\omega^x$, and finally to the whole of $\K$. It is not difficult to show that this construction always yields an exponential field. We now need to show that there is at least one function $h$ such that the exponential field $\K$ arising from $\omega$ and $h$ as above is a model of $T_{\exp}$. A necessary condition is that the exponential map grows faster than any polynomial, or equivalently, that its inverse $\log$ grows slower than $x^{1/n}$ for all positive $n\in \N$. This translates into the condition $h(x) < r \cdot \omega^x$ for every $x\in \K$ and $r \in \R^{>0}$. We shall abbreviate the above with $h(x) \prec \omega^x$.

Since $\omega^x$ is discontinuous (its values are the respresentatives of the archimedean classes), and $h$ is continuous in the order topology of $\K$ (being a chain isomorphism from $\K$ to $\K^{>0}$), the existence of such an $h$ is not immediate. In the case of Gonshor's $h$ on the surreal numbers \cite{Gonshor1986}, the condition $h(x)\prec \omega^x$ is forced by the inductive definition of $h$. However, this cannot be generalized to our more general setting where similar inductive definitions make no sense, and we use instead a bootstrapping procedure (\prettyref{lem:h-exists}). Granted this, the final $\exp$ on $\K$ is easily seen to yield a model of $T_{\exp}$ using \cite{Ressayre1993,Dries1994} (\prettyref{thm:omega-exp}).

All the logarithms considered in this paper are \textbf{analytic} (\prettyref{defn:analytic-log}): for $\eps \in o(1)$, the function $x \mapsto \log(1+x)$ is given by the familiar Taylor expansion $\log(1+\eps) = \sum_{n=1}^\infty (-1)^{n+1} \frac{\eps^n}{n}$, whereas for $g \in G$, $\log(g)$ is a \emph{purely infinite} element of $\R((G))_\kappa$, and for $r \in \R$, $\log(r)$ is the usual real logarithm. 

Theorems \ref{thm:omega-log} and \ref{thm:omega-exp} produce analytic logarithms satisfying two additional restrictions: $\log(\omega^{\omega^x}) \in G$ for all $x \in \K$, and $\log$ brings ``infinite products'' to ``infinite sums''. It turns out, however, that all analytic logarithms arise in this way, up to changing the omega-map $\omega : \K \cong G$. More precisely, we have the following classification result. 

\begin{thm*}[\prettyref{cor:criterion-for-omega}]
  Every analytic logarithm on an omega-field of the form $\K = \R((G))_\kappa$ arises from some omega-map $x\mapsto \omega^x$ and some chain isomorphism $h : \K \cong \K^{>0}$.
\end{thm*}
The surreal numbers fit into the above picture if we allow $\kappa$ to be the proper class of all ordinals and $G$ to be the image of Conway's omega-map $x\mapsto \omega^x$. Gonshor's exponentiation is induced by the omega-map and Gonshor's function $h$ \cite{Gonshor1986}; by the above results, any other analytic logarithm on $\no$ arises in this way, possibly after replacing Conway's omega-map with another isomorphism from $\no$ to its group of monomials and Gonshor's $h$ with another chain isomorphism.

\section{Preliminaries}

\subsection{Valuations}
Let $\K$ be an ordered field (possibly with additional structure) and let $\vr \subseteq \K$ be a convex subring. Then $\vr$ is the valuation ring of a valuation $v$ and we denote by $o(1)$ the unique maximal ideal of $\vr$.
If $\K$ is real closed, it has a subfield $\rf\subseteq\K$ isomorphic to the residue field $\vr/o(1)$ of the valuation, namely we can write $\vr=\rf+o(1)$. We shall always assume in the sequel that $\K$ is real closed and $\vr,o(1),\rf$ are as above.
\begin{defn}
	For $x,y\in\K$ we define:
	\begin{itemize}
		\item $x\vleq y$ if $|x|\leq c|y|$ for some $c\in\vr$ (domination);
		\item $x\veq y$ if $x\vleq y$ and $y\vleq x$ (comparability);
		\item $x\vless y$ if $x\preceq y$ and $x\not\asymp y$ (strict domination);
		\item $x\sim y$ if $x-y\prec x$ ($x$ is asymptotic to $y$).
	\end{itemize}
\end{defn}
With these notations we have $\vr=\{x\suchthat x\vleq1\}$ and $o(1)=\{x\suchthat x\vless1\}$.
\begin{defn}
	A multiplicative subgroup $G$ of $\K^{<0}$ is a group of \textbf{monomials} if it consists in a family of representatives for each $\veq$-class. In other words a group of monomials is an embedded copy of the value group. It is well known that any real closed field admits a group of monomials.
\end{defn}

\begin{rem}
	For $x,y\in \K$ we have:
	\begin{itemize}
		\item $x\prec y$ if and only if $c|x|<|y|$ for all $c\in\vr$ (or
		equivalently for all $c\in\rf$);
		\item $x\veq y$ if and only if $x=cy(1+\eps)$ for some $c\in\rf^{\times}$ and
		$\eps\in o(1)$;
		\item $x\sim y$ if and only if $x=y(1+\eps)$ for some $\eps\in o(1)$.
		\item if $x\neq 0$ there are unique $r\in \rf^{\times},g\in G,\eps\in o(1)$ such that $x=gr(1+\eps)$.
	\end{itemize}
\end{rem}

\subsection{Hahn groups\label{sec:Hahn-groups}}

By a \textbf{chain} we mean a linearly ordered set. We describe a well
known procedure to build an ordered group starting from a chain.
\begin{defn}
  Given a chain $\Gamma$ and an ordered abelian group $(C,+,<)$, the \textbf{$\Gamma$-product of} $C$ is the abelian group of all functions $f:\Gamma\to C$ with \emph{reverse} well-ordered \textbf{support} $\{\gamma\in\Gamma\suchthat f(\gamma)\neq0\}$ and pointwise addition, ordered by declaring $f>0$ if $f(\gamma)>0$, where $\gamma$ is the \emph{biggest} element in the support.\footnote{Other authors prefer to use well-ordered supports, but one can pass from one version to the other reversing the order of $\Gamma$.}

  If we write $G$ in additive notation, a typical element of $G$ can be written in the form $\sum_{\gamma\in\Gamma}\gamma r_{\gamma}$, representing the function sending $\gamma\in\Gamma$ to $r_{\gamma}\in C$, while $G$ itself is denoted $\sum \Gamma C$. We prefer however to use a multiplicative notation and write $G$ as $\prod t^{\Gamma C}$ and a typical element of  $G$ as $\prod_{\gamma \in \Gamma} t^{\gamma r_\gamma}$. In this notation the multiplication is given by
\[
\left(\prod_{\gamma\in\Gamma}t^{\gamma r_{\gamma}}\right)\left(\prod_{\gamma\in\Gamma}t^{\gamma r'_{\gamma}}\right)=\prod_{\gamma\in\Gamma}t^{\gamma(r_{\gamma}+r'_{\gamma})}
\]
Since the supports are reverse well-ordered, we can fix a decreasing
enumeration $(\gamma_{i}\suchthat i<\alpha)$ of the support, where $\alpha$ is an ordinal, and write
an element of $\prod t^{\Gamma C}$ in the form
\[
f=\prod_{i<\alpha}t^{\gamma_{i}r_{i}}\in\prod t^{\Gamma C}.
\]
According to our conventions, $f>1\iff r_{0}>0$ and $t^{\gamma}>t^{\beta}\iff\gamma>\beta.$

If $\Gamma$ has only one element, we may identify $\prod t^{\Gamma C}$ with a multiplicative copy $t^{C}$ of $(C,+,<)$.

\end{defn}
	When $C=(\R,+,<)$, we obtain the \textbf{Hahn group} over
$\Gamma$, which can be characterized as a maximal ordered group with a set of archimedean classes of the same order type as $\Gamma$ \cite{Hahn1907}. Recall that two positive elements are in the same archimedean class if each of them is bounded, in absolute value, by an integer multiple of the other.

\begin{notation}
  Let $\kappa$ be a regular cardinal. If in the definition of the $\Gamma$-product we only allow supports of reverse order type $<\kappa$, we obtain the $\kappa$-bounded version
  \[
    \left(\prod t^{\Gamma C}\right)_{\kappa} \subseteq\prod t^{\Gamma C}.\]
  We shall also consider the case when $\Gamma$
  is a proper class and $\kappa=\on$, in which case $\left(\prod t^{\Gamma C}\right)_{\on}$
  is the ordered group of all functions $f:\Gamma\to C$ whose support
  is a reverse well ordered \emph{set} (rather than a reverse well ordered class).
\end{notation}

\subsection{Hahn fields}
\label{sec:assumptions}

Given a field $\rf$ and a multiplicative ordered abelian group $G$,
let $\rf((G))$ denote the Hahn field with coefficients in
$\rf$ and monomials in $G$. The underlying additive group of $\rf((G))$ coincides with the $G$-product of $\rf$: its elements are functions $f:G\to\rf$
with reverse well-ordered supports, which we write either in the form
$f=\sum_{g\in G}gr_{g}$, where $r_{g}=f(g)$, or in in the form
\[
f=\sum_{i<\alpha}g_{i}r_{i}
\]
where $\alpha$ is an ordinal, $(g_{i})_{i<\alpha}$ is a decreasing
enumeration of the support, and $r_{i}=f(g_{i})\in\rf^{*}$. Addition
is defined componentwise and multiplication is given by the usual
Cauchy product. We order $\rf((G))$ according to the sign of the
leading coefficient, namely $f>0\iff r_{0}>0$.
\begin{rem}
  It can be proved that if $\rf$ is real closed
  and $G$ is divisible, then $\rf((G))$ is real closed \cite{Kaplansky1942}.
  Moreover, $\rf((G))$ is \textbf{spherically complete}: any decreasing
  intersection of valuation balls has a non-empty intersection.
\end{rem}
\begin{notation}
	Inside $\rf((G))$, we let $\vr$ be the valuation ring of all the elements
	$x$ with $|x|\leq r$ for some $r\in\rf$, and $o(1)$ be the
	corresponding maximal ideal. We then have $\vr=\rf+o(1)$. With respect to this valuation ring,
	$\rf$ is a copy of the residue field and $G$ is a group of monomials. We shall use similar  notations for any subfield $\K\subseteq \rf((G))$ containing $\rf$ and $G$.
\end{notation}

\subsection{Restricted analytic functions}

A family $(f_{i})_{i\in I}$ of elements of $\rf((G))$
is \textbf{summable} if the union of the supports of the elements
$f_{i}$ is reverse well-ordered and, for each $g\in G$, there are
only finitely many $i\in I$ such that $g$ is in the support of $f_{i}$.
In this case $\sum_{i\in I}f_{i}$ is defined as the element $f=\sum_g gr_g$ of $\rf((G))$ whose coefficients are given by $r_g=\sum_{i\in I}r_{g,i}\in\rf$ where $r_{g,i}$ is
the coefficient of $g$ in $f_{i}$. This makes sense since,  given $g\in G$, only finitely many $r_{g,i}$ are non-zero.

By Neumann's lemma \cite{Neumann1949} for any power series $P(x)=\sum_{n\in\N}a_{n}x^{n}$
with coefficients in $\rf$ and for any $\eps\prec1$ in $\rf((G))$,
the family $(a_{n}\eps^{n})_{n\in\N}$ is summable, so we can evaluate
$P(x)$ at $\eps$ obtaining an element $P(\eps)=\sum_{n\in\N}a_{n}\eps^{n}$
of $\rf((G))$. Similar considerations apply to power series in several variables.

\begin{defn}\label{defn:analytic subfield}
	Let $\K\subseteq\rf((G))$ be a subfield. We say that $\K$ is an
	\textbf{analytic subfield} if
	\begin{enumerate}
		\item $\K$ is truncation closed: if $\sum_{i<\alpha}g_{i}r_{i}$ belongs
		to $\K$, then $\sum_{i<\beta}g_{i}r_{i}$ belongs to $\K$ for every
		$\beta\leq\alpha$;
		\item $\K$ contains $\rf$ and $G$;
		\item If $P(x)=\sum_{n\in\N}a_{n}x^{n}$ is a power series with coefficients
		in $\rf$ and $\eps\prec1$ is in $\K$, then the element $P(\eps)=\sum_{n\in\N}a_{n}\eps^{n}\in\rf((G))$
		lies in the subfield $\K$. Similarly for power series in several variables.
\end{enumerate}
\end{defn}

We recall that $T_{an}$ is the theory of the real field with all analytic functions restricted to the poly-intervals $[-1,1]^n\subseteq \K^n$ \cite{Dries1994}. (By rescaling, we can equivalently use any other closed poly-interval.)

\begin{fact}\label{fact:an}
We have:
\begin{enumerate}
	\item The field $\R((G))$ admits a natural intepretation of the analytic functions restricted to the poly-interval $[-1,1]^n\subseteq \K$, making $\K$ into a model of $T_{an}$.
	\item The same holds for any analytic subfield of $\R((G))$, and in particular for $\R((G))_\kappa$ for every regular uncountable $\kappa$.
	\item More generally, if $\rf$ is a model of $T_{an}$, then any analytic subfield $\K$ of $\rf((G))$ is naturally a model of $T_{an}$.
\end{enumerate}
\end{fact}
The proof of (1) is in \cite{Dries1994} and is based on a quantifier elimination result in the language of $T_{an}$. The other points follow by the same argument. We interpret the restricted analytic functions in the analytic subfield $\K\subseteq \rf((G))$ as follows. Given a real analytic function $f$ converging on a neighbourhood of $[-1,1]^n \cap \R^n$, we need to define $f(x+\eps)$ where $x\in [-1,1]^n \cap \rf^n$ and $\eps\in o(1)^n \subseteq \K^n$. We do this by using the Taylor expansion $f(x+\eps) = \sum_i \frac{D^i f(x)}{i!} \eps^i$ where $i$ is a multi-index in $\N^n$. Here $D^i f(x)\in \rf$ (using the fact that $\rf$ is a model of $T_{an}$) and the infinite sum is in the sense of the Hahn field $\rf((G))$.

\subsection{Exponential fields}
A \textbf{prelogarithm} on a real closed field $\K$ is a morphism from $(\K^{>0},\cdot,1,<)$ to $(\K,+,0,<)$ and a \textbf{logarithm} is a surjective prelogarithm. An \textbf{exponential map} is the compositional inverse of a logarithm, that is an isomorphism from $(\K,+,0,<)$ to $(\K^{>0},\cdot,1,<)$. We say that $\K$ is an \textbf{exponential} field if it has an exponential map. Given a logarithm $\log$, we write $\exp$ for the corresponding exponential map and we write $e^x$ instead of $\exp(x)$ when convenient. Now assume $\rf$  has a logarithm and consider the Hahn field $\rf((G))$. It turns out that if $G\neq1$, $\rf((G))$ never admits a logarithm extending  that on $\rf$ \cite{Kuhlmann1997}. On the other hand if $\kappa$ is a regular uncountable cardinal, then for suitable choices of $G$, the logarithm on $\rf$ can be extended to $\rf((G))_\kappa$, and when $\rf = \R$ this can be done in such a way that $\rf((G))_\kappa$ is a model of $T_{\exp}$ \cite{Kuhlmann2005}.

\begin{defn}
	\label{defn:analytic-log}Let $\rf$ be an exponential
	field and let $\K$ be an analytic subfield of $\rf((G))$, for instance $\K = \rf((G))_\kappa$ with $\kappa$ regular uncountable. An \textbf{analytic
		logarithm} on $\K$ is a logarithm $\log:\K^{>0}\to\K$ with the following
	properties:
	\begin{enumerate}
		\item $\log:\K^{>0}\to\K$ extends the given logarithm on $\rf$.
		\item $\log(1+\eps)=\sum_{i=1}^{\infty}\frac{(-1)^{i+1}}{i}\eps^{i}$ for
		all $\eps\prec1$ in $\K$ (the assumption $\eps\prec 1$ ensures the summability).
		\item $\log(G)=\K^{\uparrow}$, where $\K^{\uparrow}:=\rf((G^{>1}))\cap\K$
		is the group of \textbf{purely infinite} elements, namely the
		series of the form $\sum_{i<\alpha}g_{i}r_{i}$ with $g_{i}\in G^{>1}$
		for all $i$.
	\end{enumerate}
\end{defn}

Conditions (1) and (2) are rather natural, and ensure that the restrictions of $\log(1+x)$ to small finite intervals agree with the natural $T_{an}$-interpretations of \prettyref{fact:an}. A motivation for (3) is the following. The multiplicative group
	$\K^{>0}$ admits a direct sum decomposition
	\[
	\K^{>0}=G\rf^{>0}(1+o(1)),
	\]
	namely any element $x$ of $\K^{>0}$ can be uniquely written in the
	form $x=gr(1+\eps)$ where $r\in\rf^{>0}$, $g\in G$ and $\eps\in o(1)$.
	Applying $\log$ to both sides of the above equation, we get (by (1) and (2)) a direct sum
	decomposition
	\[
	\K=\log(G)\oplus\rf \oplus o(1)
	\]
	of the additive group $(\K,+)$. Indeed
	by (1) we have $\log(\rf^{>0})=\rf$ and $\log(\K^{>0}) = \K$, while (2) implies that the logarithm maps $1+o(1)$ bijectively to $o(1)$ with inverse given by
	$
	\exp(\eps)=\sum_{n\in\N}\frac{\eps^{n}}{n!}
	$.
	We have thus proved that $\log(G)$ is a direct complement of
	$\vr=\rf+o(1)$. Although there are several choices for such a complement, the most natural one is $\log(G)=\K^{\uparrow}$, as required in point (3), since it is the unique one closed under truncations.

\subsection{Growth axiom and models of $T_{\exp}$}
Ressayre proved in \cite{Ressayre1993} that an exponential field is a model of $T_{\exp}$ if and only it satisfies the elementary properties of the real exponential restricted to $[0,1]$ and satisfies the growth axiom scheme $x\geq n^2 \to \exp(x)>x^n$ for all $n\in \N$.
\begin{defn}
	\label{defn:growth}Given an analytic subfield $\K\subseteq\rf((G))$, we say that an analytic logarithm $\log:\K^{>0}\to \K$ satisfies the \textbf{growth axiom at infinity} if $\log(x)<x^{1/n}$
	for all $x>\rf$ and all positive integers $n$.
\end{defn}
\begin{prop}\label{prop:GAlog}
	\label{fact:an-exp}If $\rf$ is a model of $T_{an,exp}$ (for instance $\rf=\R$) and $\K\subseteq\rf((G))$
	is an analytic subfield with an analytic logarithm satisfying the
	growth axiom at infinity, then $\K$ (with the natural intepretation of the symbols) is a model of $T_{an,\exp}$.
\end{prop}
\begin{proof} This follows from \cite{Ressayre1993,Dries1994} but we include some details.
	The inverse $\exp$ of an analytic logarithm is easily seen to satisfy $e^{\eps} = \sum_{n=0}^{\infty} \frac{\eps^n}{n!}$ for all $\eps\in o(1)$. Since moreover $\exp$ extends the given exponential on $\rf$, it follows that the restriction of $\exp$ to $[-1,1]$ agrees with the natural $T_{an}$-interpretation of \prettyref{fact:an}. This shows that $\K$ is at least a model of $T_{\exp|[-1,1]}$, as it is in fact the restriction of a model of $T_{an}$ to a sublanguage. Since the interpretation of $\exp$ grows faster than any polynomial (by the growth axiom at infinity plus the fact that $\rf$ is a model of $T_{\exp}$), we can conclude by the axiomatisations of \cite{Ressayre1993,Dries1994}.
\end{proof}

The above result rests on the quantifier elimination result for
$T_{an,\exp}$. We do not know whether it suffices that $\rf$ is
a model of $T_{\exp}$ to obtain that $\rf((G))_\kappa$ is a model of $T_{\exp}$ (or even $T_{\exp|[0,1]}$).

\section{Omega fields}
\begin{defn}
	A real closed field $(\K,+,\cdot,<)$ with a convex valuation ring $\vr$ and corresponding group of monomials $G\subseteq \K^{>0}$ is an {\bf omega-field} if $(\K,+,<)$ is isomorphic to $(G,\cdot, <)$ as an ordered group. Given an omega-field $\K$ we shall call \textbf{omega-map} any isomorphism of ordered groups $$\omega:(\K,+,0,<)\cong(G,\cdot,1,<).$$
\end{defn}
Since the group $G$ of monomials is isomorphic to the value group of $\K$, we have that $\K$ is an omega-field if and only if $(\K,+,<)$ is isomorphic to its value group. The definition of omega-map is inspired by Conway's omega map $\omega^x$ on the surreal numbers. We recall that the surreals can be presented in the form $\no=\R((\omega^{\no}))_{\on}$, with the image of the omega-map being the group $\omega^{\no}$ of monomials. The subscript $\on$ indicates that we only consider series whose support is a set, rather than a proper class. The surreals should thus be considered as a bounded Hahn field rather than a full Hahn field.

\subsection{Construction of omega-fields}
In the sequel let $\kappa$ be a regular uncountable cardinal.
\begin{thm}\label{thm:constructing-omega-fields}
	Given an exponential field $\rf$, there is a group $G$ such that the field $\K = \rf((G))_\kappa$ admits an omega-map $\omega:\K\to G$.
\end{thm}
When $\rf=\R$ one can take $G = \no(\kappa)$ as in \cite{DriesE2001}. In the general case the proof is a variant of a similar construction in \cite{Kuhlmann2005}. Given a chain $\Gamma$ and an additive ordered group $C$ (in our application $C=(\rf,+,<)$), let $H(\Gamma)$ denote, in the following Lemma, the ordered group $\left(\prod t^{\Gamma C} \right)_{\kappa}$.
\begin{lem}\label{lem:extending-eta}
Fix a chain $\Gamma_{0}$ and a chain embedding $\eta_{0}:\Gamma_{0}\to H(\Gamma_{0})$ (for instance $\eta_0(\gamma) = t^\gamma$). Then there is a chain $\Gamma\supseteq \Gamma_0$
	and a chain isomorphism
	$
	\eta:\Gamma\cong H(\Gamma)
	$
	extending $\eta_0$.
\end{lem}
\begin{proof}
	We consider $H$ as a functor from chains to ordered abelian groups: if $j : \Gamma' \to \Gamma''$ is a chain embedding, we define $H(j) : H(\Gamma') \to H(\Gamma'')$ as the group embedding which sends $\prod_{i}t^{\gamma_{i}r_{i}}$
	to $\prod_{i}t^{j(\gamma_{i})r_{i}}$. We do an inductive construction in $\kappa$-many steps. At a certain stage
	$\beta<\kappa$ we are given
	\[
	G_{\beta}=H(\Gamma_{\beta})
	\]
	and a chain embedding
	$\eta_{\beta}:\Gamma_{\beta}\to G_{\beta}$
	together with embeddings $j_{\alpha,\beta}:\Gamma_\alpha \to \Gamma_\beta$ for $\alpha<\beta$.
	Let $\Gamma_{\beta+1}$ be a chain isomorphic to $(G_{\beta},<)$ (for instance $G_\beta$ itself considered as a chain)
	and fix a chain isomorphism $f_{\beta}:G_{\beta}\to\Gamma_{\beta+1}$.
	Now let $j_{\beta}:\Gamma_{\beta}\to\Gamma_{\beta+1}$ be the composition
	$f_{\beta}\circ\eta_{\beta}$ and let $G_{\beta+1}=H(\Gamma_{\beta+1})_{\kappa}$.
	We can then find a commutative diagram of embeddings
	\begin{equation}
	\xymatrix{\Gamma_{\beta}\ar[d]_{j_{\beta}}\ar[r]_{\eta_{\beta}} & H(\Gamma_{\beta})\ar[d]^{H(j_{\beta})}\ar[dl]_{f_{\beta}}\\
		\Gamma_{\beta+1}\ar[r]^{\eta_{\beta+1}} & H(\Gamma_{\beta+1}),
	}
	\label{eq:commutative}
	\end{equation}
	by letting $\eta_{\beta+1}=H(j_{\beta}) \circ f_{\beta}^{-1}$. We can now define $j_{\beta,\beta+1}=j_\beta$ and $j_{\alpha,\beta+1}= j_{\beta,\beta+1}\circ j_{\alpha,\beta}$ for $\alpha<\beta$.

	We iterate the contruction along the ordinals. At a limit stage $\lambda$, let $\Gamma_{\lambda}=\varinjlim_{\beta<\lambda}\Gamma_{\beta}$ and let $j_{\beta,\lambda} : \Gamma_{\beta} \to \Gamma_{\lambda}$ be the natural embedding for $\beta < \lambda$.

	We then define $\eta_{\lambda}:\Gamma_{\lambda} \to  H(\Gamma_{\lambda})$ as the composition
	\[
	\Gamma_{\lambda} = \varinjlim_{\beta<\lambda} \Gamma_{\beta} \to \varinjlim_{\beta<\lambda}H(\Gamma_\beta) \to H(\varinjlim_{\beta<\lambda}\Gamma_{\beta}) = H(\Gamma_{\lambda}).
	\]
	More explicitly, for each $\gamma\in\Gamma_{\lambda}$, pick some
	$\beta<\lambda$ and $\theta\in\Gamma_{\beta}$ such that $\gamma=j_{\beta,\lambda}(\theta)$,
	and define $\eta_{\lambda}(\gamma)\in G_{\lambda}$ as the
	image under $H(j_{\beta,\lambda}):G_{\beta}\to G_{\lambda}$ of $\eta_{\beta}(\theta)\in G_{\beta}$.
	Since $\kappa$ is regular, when we arrive at stage $\kappa$ we have
	an isomorphism
	\[
	\eta_{\kappa}:\Gamma_{\kappa}\cong G_{\kappa}
	\]
	and we can define $\Gamma=\Gamma_{\kappa}$
	and $\eta=\eta_{\kappa}$.
\end{proof}

\begin{proof}[Proof of \prettyref{thm:constructing-omega-fields}]
	By \prettyref{lem:extending-eta}, there is a chain $\Gamma$ and a chain isomorphism
	\begin{equation}
	\eta:\Gamma\cong G=H(\Gamma)
	\end{equation}
	Now let $\K = \rf((G))_\kappa$ and define an omega-map $\omega:\K\to G$ by setting
	\begin{equation}
	\omega^{\sum_{i<\alpha}g_{i}r_{i}}=\prod_{i<\alpha}t^{\gamma_i r_{i}}.
	\end{equation}
	where $g_i = \eta(\gamma_i)$. In particular $\omega^{\eta(\gamma)} = t^\gamma$ for every $\gamma\in \Gamma$.
\end{proof}

\subsection{The logarithm}
In the sequel let $\kappa$ be a regular uncountable cardinal.
Our next goal is to prove the following theorem.

 \begin{thm}
	\label{thm:omega-log} Every omega-field of the form $\K = \R((G))_\kappa$ admits an analytic logarithm. More generally, if $\rf$ is an exponential field, then every omega-field of the form $\K = \rf((G))_\kappa$ admits an analytic logarithm.
\end{thm}

\begin{proof}
	We construct a logarithm depending both on the omega-map and on an auxiliary function $h$.
	Let $h:\K\to\K^{>0}$ be a chain isomorphism (any ordered field admits such a function, for instance $h(x)=(-x+1)^{-1}$ for $x\leq0$
	and $h(x)=x+1$ for $x\geq0$). For $x\in\K$, we let
	\[
	\log(\omega^{\omega^{x}})=\omega^{h(x)}.
	\]
	This defines $\log$ on the subclass  $\omega^{\omega^{\K}}$ of $G$, which we call the class
	of \textbf{fundamental monomials}. They can be seen as the representatives of the multiplicative archimedean classes.

	Next we define $\log(g)$ for an arbitary $g$ in $G$. Since $G=\omega^{\K}$,
	we can write $g=\omega^{x}$ for some $x\in\K$. We then write
	$x=\sum_{i<\alpha}g_{i}r_{i}=\sum_{i<\alpha}\omega^{x_{i}}r_{i}$
	and set $\log(g)=\sum_{i<\alpha}\omega^{h(x_{i})}r_{i}$. Summing
	up, the definition of $\log_{|G}$ takes the form
	\begin{equation}
	  \label{eq:log-monomial} \log\left(\omega^{\sum_{i<\alpha}\omega^{x_{i}}r_{i}}\right)=\sum_{i<\alpha}\omega^{h(x_{i})}r_{i}.
    \end{equation}
	The idea is that $\omega^{\sum_{i<\alpha}g_{i}r_{i}}$ should be thought
	as an infinite product $\prod_{i<\alpha}\omega^{g_{i}r_{i}}$, and we stipulate that $\log$ maps infinite products into infinite sums.

	We can now extend $\log$ to the whole of $\K^{>0}$ as follows. For
	$x\in\K^{>0}$ we write
	$
	x=gr(1+\eps)
	$
	with $g\in G,r\in\rf^{>0}$ and $\eps\prec1$, and we define
	\begin{equation}
	  \label{eq:log-x}
	  \log(x)=\log(g)+\log(r)+\sum_{n=1}^{\infty}(-1)^{n+1}\frac{\eps^{n}}{n}.
	\end{equation}
	The infinite sum makes sense because the terms under the summation sign are summable and the sum belongs to $\rf((G))_\kappa$ (becacuse $\kappa$ is regular and uncountable).

	We must verify that with these definitions $\log$ is an analytic logarithm (\prettyref{defn:analytic-log}). It is not difficult
	to see that $\log$ is an increasing morphism from $(\K^{>0},\cdot,1,<)$
	to $(\K,+,0,<)$.  To prove the surjectivity let us first observe that $\rf=\log(\rf^{>0}) \subseteq \log(\K^{>0})$. Moreover,  for $\eps\prec 1$ we have $\log(1+\eps)=\sum_{n=1}^{\infty}(-1)^{n+1}\frac{\eps^{n}}{n}$ with inverse given by $e^{\eps} = \sum_n \frac{\eps^n}{n!}$,
	and therefore $\log(1+o(1))=o(1)$.
	Since $\K=\K^{\uparrow}+\rf+o(1)$, to finish the proof of the
	surjectivity it suffices to show that $\log(G)=\K^{\uparrow}$.
	So let $x=\sum_{i<\alpha}g_{i}r_{i}\in\K^{\uparrow}$, namely $g_{i}\in G^{>1}$
	for all $i$. We must show that $x$ is in the image of $\log$. Since $h:\K\to\K^{>0}$ is surjective and $G=\omega^{\K}$, we
	have $G^{>1}=\omega^{(\K^{>0})}=\omega^{h(\K)}$, so we can choose
	$x_{i}\in\K$ so that $g_{i}=\omega^{h(x_{i})}$ for all $i$. Now by definition
	$\log\left(\omega^{\sum_{i<\alpha}\omega^{x_{i}}r_{i}}\right)=\sum_{i<\alpha}\omega^{h(x_{i})}r_{i} = x$ concluding the proof of surjectivity.
\end{proof}

In the above theorem we have considered $\rf((G))_\kappa$, rather than an arbitrary analytic subfield $\K$  of $\rf((G))$, because for the proof to work we need to know that whenever $\sum_{i<\alpha} \omega^{x_i} r_i \in \K$, we also have $\sum_{i<\alpha} \omega^{h(x_i)} r_i \in \K$.

\begin{defn}\label{defn:omega-log}
  We call $\log_{\omega,h} : \K^{>0} \to \K$ the analytic logarithm induced by the omega-map $\omega : \K \to G$ and the chain isomorphism $h : \K \to \K^{>0}$ as given by \prettyref{eq:log-monomial}-\prettyref{eq:log-x} in the proof of \prettyref{thm:omega-log}, and we call $\exp_{\omega,h}$ its compositional inverse.
\end{defn}

\subsection{Getting a logarithm satisfying the growth axiom}
The structures constructed so far are exponential fields, but not necessarily models of $T_{\exp}$. In this section we show how to get models of $T_{\exp}$. We need the following lemma to take care of the growth axiom at infinity.

\begin{lem} \label{lem:h-exists}
	Let $\K = \rf((G))_{\kappa}$ be equipped with an omega-map $\omega : \K \cong G$. Then there exists a chain isomorphism $h : \K \to \K^{>0}$ such that $h(x) \prec \omega^x$ for all $x \in \K$.
\end{lem}

\begin{proof}
	The idea is a bootstrapping procedure. Given an $h$ we produce a $\log$ and an $\exp$, and given the $\exp$ we produce a new $h$. We then glue together a couple of $h$'s obtained in this way to produce the final $h$.

	To begin with, consider the following chain isomorphism $\K \to \K^{>0}$, definable in any ordered field:
	\[
	h_0(x) =
	\begin{cases}
	x + 1 & \textrm{for } x \geq 0 \\
	\frac{1}{1-x} & \textrm{for } x < 0,

	\end{cases} \quad \quad \quad
	h_1(x) =
	\begin{cases}
	\frac{1}{2}x + 1 &  \textrm{for } x \geq 0 \\
	\frac{1}{1-x} &  \textrm{for } x < 0.
	\end{cases}
	\]
	\prettyref{defn:omega-log} yields two logarithmic functions $\log_0 = \log_{\omega,h_0}$ and $\log_1 = \log_{\omega,h_1}$ on $\rf((G))_{\kappa}$ associated with $h_0$ and $h_1$ (and the given omega-map). Since $h_1(x)\leq h_0(x)$, we have $\log_1(x) \leq \log_0(x)$ for all $x\in \K^{>1}$.  The corresponding exponential functions $\exp_0, \exp_1$ satisfy the opposite inequality: $\exp_0(x)\leq \exp_1(x)$ for $x > 0$.

	We claim that
	\[ \exp_0(x) \prec \omega^x\ \textrm{for}\ x > \rf \quad \textrm{and } \quad \exp_1(x) \prec \omega^x\ \textrm{for}\ x \leq - \omega^3.\]Indeed,
	note that $h_0(x) > x$ for all $x \in \K$ and $h_1(x)<x$ for $x>2$. Taking the compositional inverse we obtain $x>h_0^{-1}(x)$ for all $x\in \K$ and $x < h_1^{-1}(x)$ for $x>2$.
	Now let $y \in \K^{> \rf}$, and let $r\omega^x$ be the leading term of $y$ (where $r \in \rf^{>0}$, $x \in \K^{>0}$). Then
	\[ \exp_0(y) \prec \exp_0\left(2r\omega^x\right) = \omega^{2r\omega^{h_0^{-1}(x)}} \prec \omega^{\frac{r}{2}\omega^x} \prec \omega^y,\]
	since $2r\omega^x - y > \rf$, $y - \frac{r}{2}\omega^x > \rf$, and $\omega^{h_0^{-1}(x)} \prec \omega^x$.

	Similarly, $h_1(x) < x$ for all $x \in \K^{>2}$.  Let $y \in \K^{\geq \omega^3}$, and let $r\omega^x$ be the leading term of $y$. Then $r \in \rf^{>0}$, $x \in \K^{>2}$ and
	\[ \exp_1(y) \succ \exp_1\left(\frac{r}{2}\omega^x\right) = \omega^{\frac{r}{2}\omega^{h_1^{-1}(x)}} \succ \omega^{2r\omega^x} \succ \omega^y.\]
	Letting $z = -y \leq -\omega^3$, we obtain $\exp_1(z) = \frac{1}{\exp_1(y)} \prec \frac{1}{\omega^y} =\omega^z$, and the claim is proved.

	We can now build the final chain isomorphism $h : \K \to \K^{>0}$ by taking the functions $\exp_0$, $\exp_1$ restricted to suitable convex subsets of $\K$, and defining $h$ on the complement as an increasing function in such a way that globally $h$ is increasing and bijective. A concrete choice can be the following. Let $c = \exp_1(-\omega^3) > 0$. Define
	\[ h(x) =
	\begin{cases}
	\exp_0(x) & \text{for } x > \rf \\
	2c + x & \text{for } 0 < x \textrm{ and } x \preceq 1 \\
	2c + \frac{c}{\omega^3}x &\text{for } -\omega^3 \leq x \leq 0\\
	\exp_1(x) &\text{for } x < -\omega^3.
	\end{cases}
	\]
	By construction, $h$ is a chain isomorphism $h : \K \to \K^{>0}$: it is order preserving because $\exp_0$, $\exp_1$ are themselves chain isomorphisms, and it is surjective since $\exp_0(\K^{>\rf}) = \K^{>\rf}$, $\exp_1((-\infty,-\omega^3)) = (0,c)$. Moreover, $h(x) \prec \omega^x$ for all $x \in \K$, as desired:
	\begin{itemize}
		\item if $x > \rf$, then $h(x) = \exp_0(x) \prec \omega^x$;
		\item if $0< x \preceq 1$, then $h(x) = 2c + x \preceq 1 \prec \omega^x$;
		\item if $-\omega^3 \leq x \leq 0$, then $h(x) \asymp c = \exp_1(-\omega^3) \prec \omega^{-\omega^3} \preceq \omega^x$;
		\item if $x < -\omega^3$, then $h(x) = \exp_1(x) \prec \omega^x$. \qedhere
	\end{itemize}
\end{proof}

We next show that an $h$ as constructed above is sufficient to guarantee the growth axiom at infinity.

\begin{lem} \label{lem:GAh} Let $\log=\log_{\omega,h}:\K^{>0}\to \K$ be as in
	\prettyref{defn:omega-log}.  If $h$ satisfies $h(x)\prec\omega^{x}$ for every $x\in\K$, then $\log(y)<y^{r}$ for all positive $r\in \rf$ and all $y>\rf$ (where $y^{r}$ is defined as $e^{r\log(y)}$).
\end{lem}
\begin{proof}
	Assume $h(x)\prec\omega^{x}$ for all $x\in\K$. This means that $h(x)<\omega^{x}r$
	for all $r\in\rf^{>0}$. Let $y=\omega^{\omega^{x}}$. Then $\log(y)=\log(\omega^{\omega^{x}})=\omega^{h(x)}<\omega^{\omega^{x}r}=y^{r}$. We have thus proved that $\log(y)<y^{r}$
	for $y$ of the form $\omega^{\omega^{x}}$ and $r\in\rf^{>0}$.

	We now prove the inequality for $y$ of the form $\omega^{x}$, where
	$x\in\K^{>0}$. To this aim we write the exponent $x$ in the form
	$\sum_{i<\alpha}\omega^{x_{i}}r_{i}$ and observe that $r_{0}>0$
	and $\log(\omega^{x})=\log\left(\omega^{\sum_{i<\alpha}\omega^{x_{i}}r_{i}}\right)=\sum_{i<\alpha}\log(\omega^{\omega^{x_{i}}})r_{i}$.
	By the special case we have $\log(\omega^{\omega^{x_{i}}})<\omega^{\omega^{x_{i}}a}\leq\omega^{\omega^{x_{0}}a}$ for every $i<\alpha$ and $a\in\rf^{>0}$.  Letting $a = rr_0/2$ it follows that
	\[
	\log(\omega^{x})<\omega^{\omega^{x_{0}}a}=\left(\omega^{\omega^{x_{0}}r_{0}}\right)^{\frac{a}{r_{0}}}<\left(\omega^{2x}\right)^{\frac{a}{r_{0}}}=\omega^{xr}.
	\]
	For a general $y>\rf$, write $y$ in the form $\omega^{x}s(1+\eps)$
	with $s\in\rf^{>0}$, $x>0$ and $\eps\prec1$, and observe that
	$
	\log(y)<\log(2s)+\log(\omega^{x})<(\omega^{x})^{\frac{r}{2}}<y^{r}
	$
	for any $r\in\rf^{>0}$.
\end{proof}

	In the case when the residue field $\rf$ is archimedian, the statement in the conclusion of \prettyref{lem:GAh} is equivalent to the growth axiom at infinity (\prettyref{defn:growth}). We are now ready for the main result of this section.

\begin{thm}\label{thm:omega-exp}
    Every omega-field of the form $\K = \R((G))_\kappa$ admits an analytic logarithm making it into a model of $T_{an,\exp}$. More generally, if $\rf$ is a model of $T_{an,\exp}$, then every omega-field of the form $\K = \rf((G))_\kappa$ admits an analytic logarithm making it into a model of $T_{an,\exp}$.
\end{thm}

\begin{proof}By \prettyref{prop:GAlog} and \prettyref{lem:GAh}.
\end{proof}

\subsection{Growth axiom and o-minimality}
We now discuss the connections between the growth axiom and o-minimality (see \cite{vdDries1998} for the development of the theory of o-minimal structures).
\begin{lem}\label{lem:omin}
  Let $\K$ be an o-minimal exponential field. Note that $\exp$ must be differentiable and by a linear change of variable, we can assume that $\exp'(0)=1$. Then $\exp(x)>x^n$ for all positive $n\in \N$ and all $x>\N$.
\end{lem}
\begin{proof}
  Given a definable differentiable unary function $f:\K\to \K$ in an o-minimal expansion of a field, its derivative $f'$ is definable, and if $f'$ is always positive, then $f$ is increasing. It follows that if $f,g$ are definable differentiable functions satisfying $f(a)\leq g(a)$ and $f'(x)<g'(x)$ for all $x\geq a$, then $f(x)<g(x)$ for every $x> a$.
  Starting with $0 < \exp(x)$ and integrating we then inductively obtain that for each positive $k,n\in \N$ there is a positive $c\in \N$ such that $kx^n \leq e^x$ for all $x>c$.
\end{proof}
By the above observation and Ressayre's axiomatization \cite{Ressayre1993}, an exponential field is a model of $T_{\exp}$ if and only if it satisfies the complete theory of restricted exponentiation and it is o-minimal.

\begin{thm}\label{thm:omin}
Assume $\K=\R((G))_\kappa$ has an omega-map $\omega:\K\cong G$. Fix a chain isomorphism $h:\K\cong \K^{>0}$ and put on $\K$ the logarithm induced by $\omega$ and $h$ as in \prettyref{defn:omega-log}. Then $\K$ is either a model of $T_{\exp}$ or it is not even o-minimal.
\end{thm}
\begin{proof}
We have already seen that if $h(x)\prec \omega^x$ for all $x\in \K$, then $\K$ is a model of $T_{\exp}$ (\prettyref{thm:omega-exp}). Now suppose that $h(x)\not\prec \omega^x$ for some $x$. Then there is some $n \in \N^{>0}$ such that $h(x) \geq \frac{1}{n}\omega^x$. Letting $y = \omega^{\frac{1}{n}\omega^x}$, we have $\log(y) = \frac{1}{n}\log(\omega^{\omega^x}) = \frac{1}{n}\omega^{h(x)} \geq \frac{1}{n}\omega^{\frac{1}{n}\omega^x} = \frac{1}{n}y$, hence $y^n \geq e^y$, contradicting o-minimality by \prettyref{lem:omin} (since $\exp$ extends the real exponential function, we have $\exp'(0)=1$, so the hypothesis of the lemma are satisfied).
\end{proof}

\section{Other exponential fields of series}

\subsection{Criterion for the existence of an omega-map}
In this section we try to classify all possible analytic logarithms on $\rf((G))_\kappa$. We show that in the case of omega-fields every analytic logarithm arises from an omega-map and some $h$.

\begin{thm}\label{thm:criterion-for-omega}
	Assume that $\K=\rf((G))_{\kappa}$ has an analytic logarithm $\log$. Then:
	\begin{enumerate}
		\item $\K$ has an omega-map $\omega:\K\cong G$ if and only if $G$ is isomorphic to $G^{>1}$ as a chain;
		\item moreover, if $G\cong G^{>1}$, there is an omega-map and a chain isomorphism $h:\K\cong\K^{>0}$ such that the
		logarithm induced by $\omega$ and $h$ coincides with the original logarithm.
	\end{enumerate}
\end{thm}
\begin{proof}
	First note that $\K$, being an ordered field, is always isomorphic
	to $\K^{>0}$ as a chain. If there is an omega-map $\omega:\K\cong G$, we obtain an
	induced isomorphism from $G=\omega^{\K}$ to $G^{>1}=\omega^{\K^{>0}}$.

	For the opposite direction, assume that $G$ is isomorphic to $G^{>1}$
	as a chain and let $\psi:G\cong G^{>1}$ be a chain isomorphism. Define
	$\omega:\K\to G$ by
	\[
	\omega^{\sum_{i<\alpha}g_{i}r_{i}}=e^{\sum_{i<\alpha}\psi(g_{i})r_{i}}.
	\]
	In particular we have $\omega^g = e^{\psi(g)}$. Clearly $\omega$ is a morphism from $(\K,+,0,<)$ to $(G,\cdot,1,<)$
	and to prove that it is an omega-map it only remains to verify that
	it is surjective. To this aim recall that $\log(G)=\K^{\uparrow}$
	(by definition of analytic logarithm), so for the corresponding $\exp$
	we have $G=\exp(\K^{\uparrow})$. Since $e^{\sum_{i<\alpha}\psi(g_{i})r_{i}}$
	is an arbitrary element of $\exp(\K^{\uparrow})$, the surjectivity
	of $\omega$ follows. Now since $\psi:G\cong G^{>1}$ and
	$G=\omega^{\K}$, there is a chain isomorphism $h:\K\to\K^{>0}$ such
	that
	\[
	\psi(\omega^{x})=\omega^{h(x)}.
	\]
	Since $e^{\psi(\omega^x)} = \omega^{\omega^x}$, we obtain $\omega^{\omega^{x}}=e^{\omega^{h(x)}}$
	and thefore $\log(\omega^{\omega^{x}})=\omega^{h(x)}$. It then follows
	that $\log$ coincides with the analytic logarithm induced by $\omega$
	and $h$.
\end{proof}

\begin{cor}\label{cor:criterion-for-omega}
  Every analytic logarithm on an omega-field of the form $\K = \rf((G))_\kappa$ arises from some omega-map and some chain isomorphism $h : \K \cong \K^{>0}$.
\end{cor}

\subsection{The iota-map} Our next goal is to show that $\rf((G))_\kappa$ may have an analytic logarithm without being an omega-field. This will be proved in the next subsection. Here we recall the following two results from \cite{Kuhlmann2005} with a sketch of the proofs for the reader's convenience (considering that the notations are different). We use the same notation $H(\Gamma)=\left(\prod t^{\Gamma C} \right)_{\kappa}$ employed in  \prettyref{lem:extending-eta}, with $C=(\rf,+,<)$.

\begin{fact}[\cite{Kuhlmann2005}] \label{fact:iota-log} Let $\rf$ be an exponential field. Let $\Gamma$ be a chain and suppose there is an isomorphism of chains $\iota:\Gamma\cong H(\Gamma)^{>1}$. Let $G=H(\Gamma)$  and let $\K = \rf((G))_\kappa$. Then:
	\begin{enumerate}
		\item there is an analytic logarithm $\log:\K^{>0}\to \K$ such that $\log(t^\gamma) = \iota(\gamma)\in G$.
		\item if $\rf$ is a model of $T_{an,\exp}$ and $\iota(\gamma)<t^{\gamma r}$ for each $r\in \rf^{>0}$, then $\log$ satisfies the growth axiom at infinity, thus making $\K$ into a model of $T_{an,\exp}$.\footnote{In the cited paper the authors consider $\rf = \R$, but the general case is the same.}
	\end{enumerate}
\end{fact}
\begin{proof}
	Define $\log=\log_{\iota}$ on $G$ by
	\[
	\log(\prod_{i<\alpha}t^{\gamma_{i}r_{i}})=\sum_{i<\alpha}\iota(\gamma_{i})r_{i}\in\rf((G^{>1}))_{\kappa}
	\]
	Given
	$x\in\K^{>0}$, write
	$
	x=gr(1+\eps)
	$
	for some $r\in\rf^{>0}$, $g\in G$ and $\eps\in o(1)$; now define
	$
	\log(x)=\log(g) + \log(r) + \sum_{n=1}^{\infty}(-1)^{n+1}\frac{\eps^{n}}{n},
	$
	where $\log(r)$ refers to the given logarithm on $\rf$, and observe that since $\eps\prec1$ and $\kappa>\omega$ the infinite sum belongs to $\K=\rf((G))_{\kappa}$.
	Clearly $\log$ is an analytic logarithm and (1) is proved. The verification of point (2) is as in  \prettyref{thm:omega-log}.
\end{proof}
\begin{fact}[\cite{Kuhlmann2005}]\label{fact:iota}

	Fix a chain $\Gamma_{0}$ and a chain embedding $\iota_{0}:\Gamma_{0}\to H({\Gamma_{0}})^{>1}$ (for instance $\iota_0(\gamma) = t^\gamma$). Then:
	\begin{enumerate}
		\item there is a chain $\Gamma\supseteq \Gamma_0$
		and a chain isomorphism
		$
		\iota:\Gamma\cong H(\Gamma)^{>1}
		$
		extending $\iota_0$;
		\item if $\iota_0(\gamma)<t^{\gamma r}$ for every $\gamma\in \Gamma_0$ and $r\in C^{>0}$, then $\iota(\gamma)<t^{\gamma r}$ for every $\gamma\in \Gamma$ and $r\in C^{>0}$.
	\end{enumerate}
\end{fact}
\begin{proof}
	The proof of (1) is similar to the proof of \prettyref{lem:extending-eta}, the only difference is that we use $H(\Gamma)^{>1}$ instead of $H(\Gamma)$. Starting with the initial chain embedding $\iota_0:\Gamma_0 \to H(\Gamma_0)^{>1}$ we inductively produce chain embeddings  $\iota_\beta:\Gamma_\beta \to H(\Gamma_\beta)^{>1}$ and $j_{\alpha,\beta}:\Gamma_\alpha\to \Gamma_\beta$ for $\alpha<\beta$. The step from $\beta$ to $\beta+1$ is based on the following diagram
	\begin{equation}
 \xymatrix{\Gamma_{\beta}\ar[d]_{j_{\beta}}\ar[r]_{\iota_{\beta}} & H(\Gamma_{\beta})^{>1}\ar[d]^{H(j_{\beta})}\ar[dl]_{f_{\beta}}\\
		\Gamma_{\beta+1}\ar[r]^{\iota_{\beta+1}} & H(\Gamma_{\beta+1})^{>1}
	}
	\label{eq:commutative-2}
	\end{equation}
	where $\Gamma_{\beta+1}$ is a chain isomorphic to $H(\Gamma_\beta)^{>1}$, $f_\beta$ is a chain isomorphism, and the embeddings $j_\beta$ and $\iota_{\beta+1}$ are defined so that the diagram commutes.
	Limit stages are handled as in \prettyref{lem:extending-eta}.
	Finally we set $\Gamma = \Gamma_\kappa = \varinjlim_{\beta<\kappa} \Gamma_\beta$ and $\iota=\iota_\kappa$ and observe that $\iota:\Gamma\to H(\Gamma)^{>1}$ is a chain isomorphism.

	To prove (2), we show by induction on $\beta<\kappa$ that $\iota_\beta(\gamma)<t^{\gamma r}$ for every $\gamma\in \Gamma_\beta$ and $r\in C^{>0}$, provided this holds for $\beta=0$. Since limit stages are easy, it suffices to prove the induction step from $\beta$ to $\beta+1$. So let
	$\eta \in\Gamma_{\beta+1}$. Then $\eta=f_\beta(x)$ for some
	$x = \prod_{i}t^{\gamma_{i}r_{i}}\in\left(\prod t^{\Gamma_{\beta}C}\right)_{\kappa}^{>1}$. The embedding $\iota_{\beta}$
	sends $\eta$ to $\prod_{i}t^{j_\beta(\gamma_{i})r_{i}}$ where $j_\beta= f_{\beta}\circ\iota_{\beta}$
	is the embedding of $\Gamma_{\beta}$ into $\Gamma_{\beta+1}$. We
	must prove that $\prod_{i}t^{j_\beta(\gamma_{i})r_{i}}<t^{\eta r}$ for
	every $r\in C^{>0}$. This is equivalent to saying $j_\beta(\gamma_{0})<\eta$,
	which in turn is equivalent to $\iota_{\beta}(\gamma_{0})<\prod_{i}t^{\gamma_{i}r_{i}}$.
	The latter inequality follows from the inductive hypothesis
	and the proof is complete.
\end{proof}

\subsection{A model without an omega-map}
We can now show that there are fields of the form $\R((G))_\kappa$ which admit an analytic logarithm but not an omega-map.
\begin{thm} \label{thm:no-omega} Given a regular uncountable cardinal $\kappa$, there is $G$ such that the field $\K=\R((G))_{\kappa}$ has an analytic logarithm making it into a model of $T_{\exp}$ but $G$ is not isomorphic to $G^{>1}$ as a chain (so $\K$ is not an omega-field).
\end{thm}
\begin{proof}
	Start with the chain $\Gamma_{0}=\omega_{1}\times\Z$ ordered lexicographically and the initial
	embedding $\iota_{0}:\Gamma_{0}\to\left(\prod_{\kappa}t^{\Gamma_{0}\rf}\right)^{>1} = H(\Gamma_0)^{>1}$
	given by $\iota_{0}((\alpha,n))=t^{(\alpha,n-1)}$. Define $\Gamma = \varinjlim_{\beta<\kappa}\Gamma_\beta$ and $\iota:\Gamma\cong H(\Gamma)^{>1}$ as in \prettyref{fact:iota} and note that  $\iota(\gamma)<t^{\gamma r}$ for every $\gamma\in \Gamma$ and $r\in \rf^{>0}$ (since this holds for $\iota_0$ and is preserved at later stages).
	Now take $G=H(\Gamma)$ and put on the field $\K=\rf((G))_\kappa$ the $\log$ induced by $\iota$ as in \prettyref{fact:iota-log}. By the above inequalities the $\log$ satisfies the growth axiom at infinity, so $\K$ is a model of $T_{\exp}$.
	It remains to show that $G\not\cong G^{>1}$ as a chain. Note that the image of $\iota_{0}:\Gamma_{0}\to H(\Gamma_{0})^{>1}=\Gamma_1$
	is cofinal and coinitial in $H(\Gamma_{0})^{>1}$. It
	follows that for each $\beta\leq\kappa$, the image of $\iota_{\beta}:\Gamma_{\beta}\to H(\Gamma_{\beta})^{>1}$
	is cofinal and coinitial in $H(\Gamma_{\beta})^{>1}= \Gamma_{\beta+1}$. Likewise, by an easy induction, for each $\beta\geq 0$ the image of $\Gamma_0$ in $\Gamma_\beta$ is initial and cofinal.
	In particular the image of $\Gamma_{0}$ in the final chain $\Gamma_\kappa=\Gamma \cong H(\Gamma)^{>1}$
	is coinitial and cofinal. Since $\Gamma_{0}$ has cofinality $\omega_{1}$
	and coinitiality $\omega$, it follows that $\Gamma$ and $H(\Gamma)^{>1}$
	have cofinality $\omega_{1}$ and coinitiality $\omega$. Now observe
	that $1/x$ is an order-reversing bijection from $H(\Gamma)^{<1}$
	to $H(\Gamma)^{>1}$, and therefore $H(\Gamma)=H(\Gamma)^{<1}\cup1\cup H(\Gamma)^{>1}$
	has cofinality and coinitiality both equal to $\omega_{1}$. We conclude
	that $G=H(\Gamma)$ cannot be chain isomorphic to $G^{>1}$, because
	they have different coinitiality.
\end{proof}

\section{Omega-groups}
A group isomorphic to the value group of an omega-field will be called \textbf{omega-group}. It would be interesting to give a characterization of the omega-groups. As a partial result, we characterise those groups $G$ such that $\rf((G))_\kappa$ is an omega-field. We also clarify the relation between having an omega-map and having an analytic logarithm.

\begin{prop}
Let $\K$ be a field of the form $\rf((G))_\kappa$. Then:
\begin{enumerate}
	\item if $\K$ is an omega-field, then $G$ is isomorphic to $\left(\prod t^{\Gamma \rf} \right)_\kappa$, where the chain $\Gamma$ is order-isomorphic to (the underlying chain of) $G$ itself;
	\item if $\K$ has an analytic logarithm, then $G$ is isomorphic to $\left(\prod t^{\Gamma \rf} \right)_\kappa$, where $\Gamma$ is order-isomorphic to $G^{>1}$.
\end{enumerate}
\end{prop}
\begin{proof}
	(1) The elements of $\K$ can be written in the form $\sum_{i<\alpha} g_i r_i$. So the elements of $G$ are of the form $\omega^{\sum_{i<\alpha} g_i r_i}$. This corresponds to the element $\prod_{i<\alpha}t^{g_i r_i} \in \left( \prod t^{G \rf} \right)_\kappa$ via an isomorphism.

	(2) Since $\log(G)=\K^{\uparrow}$, we have $G=\exp(\K^{\uparrow})$, and therefore an element $g$ of $G$ can be written in the form $\exp(\sum_{i<\alpha}g_i r_i)$ with $g_i\in G^{>1}$ and $r_i\in \rf$. This corresponds to $\prod_{i<\alpha} t^{g_i r_i} \in \left( \prod t^{G^{>1}\rf} \right)_\kappa$ via an isomorphism.
\end{proof}

In the following corollary we abstract some of the properties of the groups considered above. We refer to \cite{Kuhlmann2000} for the definition of the value-set.

\begin{cor} Let $\K$ be a field of the form $\rf((G))_\kappa$.
\begin{enumerate}
\item If $\K$ has an analytic logarithm, then $G$ is a $\rf$-module, the value set $\Gamma$ of $G$ is order isomorphic to $G^{>1}$, and all the $\rf$-archimedean components of $G$ are isomorphic to the additive group of $\rf$.
\item If $\K$ is an omega-field, the same properties hold (as in particular $\K$ has an analytic logarithm) and in addition $G$ is isomorphic to $G^{>1}$ as a chain.
\end{enumerate}
\end{cor}

\subsubsection*{Acknowledgements}
	All the authors want to thank the organizers of the trimester ``Model Theory, Combinatorics and Valued fields in model
	theory" at the Institute Henri Poincaré, January 8 - April 6, 2018 where part of this research was conducted. We also thank the anonymous referee for useful comments.

\end{document}